\documentclass{amsart}
\usepackage[utf8]{inputenc}

\usepackage{setspace}
\usepackage{geometry}
\usepackage{enumerate}
\usepackage{enumitem, xcolor, amssymb,latexsym,amsmath,bbm,cite}
\usepackage{mathtools}
\usepackage[mathscr]{euscript}
\usepackage{amsmath}
\usepackage{amssymb}
\usepackage{enumitem}
\newcommand{\pp}{\mathbb{P}}
\newcommand{\R}{\mathbb{R}}

\newcommand{\Exp}{\mathbb{E}}
\newcommand{\F}{\mathcal{F}}
\newcommand{\Sc}{\mathcal{S}}

\newcommand{\inpr}[3][]{\left\langle#2 \,,\, #3\right\rangle_{#1}}

\theoremstyle{theorem}
\newtheorem{theorem}{Theorem}[section]

\newtheorem{lemma}[theorem]{Lemma}
\theoremstyle{proposition}
\newtheorem{proposition}[theorem]{Proposition}
\theoremstyle{definition}
\newtheorem{definition}[theorem]{Definition}
\theoremstyle{remark}
\newtheorem{remark}[theorem]{Remark}

\numberwithin{equation}{section}

\allowdisplaybreaks

\title[Invariant Measure]{Invariant Measure for Linear Stochastic PDEs in the space of Tempered distributions}
\author{Arvind Kumar Nath}

\address{Arvind Kumar Nath, Department of Mathematics and Statistics, Indian Institute of Technology Kanpur, Kalyanpur, Kanpur - 208016, India.}
\email{yarvind@iitk.ac.in}

\begin{document}
%\maketitle
\begin{abstract}
    In this paper, we first explore exponential stability by using Monotonicity inequality and use this information to obtain the existence of Invariant measure for linear Stochastic PDEs with potential in the space of tempered distributions. The uniqueness of Invariant Measure  follows from Monotonicity inequality.
\end{abstract}


\maketitle

\section{introduction}\label{s1:intro}
Let $\Sc$ be the space of real-valued smooth rapidly decreasing functions on $\R^d$ with the  topology given by Schwartz. Under this topology, $\Sc$  is a locally convex space and its  (continuous) dual $\Sc'$ is the space of  tempered distributions (see \cite{MR1681462, MR2296978, MR771478, MR1157815} and the references therein). Let $(\Omega, \F, (\F_t)_t, \pp)$ be a given filtered probability space satisfying the usual conditions (see \cite{MR1121940, MR2020294, MR2560625, MR3236753} and the references therein). In this article, we are interested in the following properties, exponential stability and invariant measures, for linear stochastic partial differential equations (SPDEs) in $\Sc^\prime$ of the form
\begin{equation} \label{SPDE}
    dX_t=(L-\alpha) (X_t)dt + A(X_t)\,dB_t,
\end{equation}
where $\{B_t=(B_t^1,\cdots,B_t^d), t \geq 0\}$ is a standard $d$-dimensional Brownian motion, $\alpha \in \R$, 
  and $L$ and $A = (A_1, A_2, \cdots, A_d)$ are certain linear differential operators defined below (see subsection \ref{S1:operators}). The expresssion $A(X_t)\,dB_t$ means $\sum_{j=1}^dA_j(X_t)\,dB_t^j$.

Exponential stability and invariant measures have been extensively studied in the literature (see \cite{MR0738933,MR1396756,MR2560625,MR0916950} and the references therein). 
Our main results, stated in the subsection \ref{S1:results}, are analogous in nature with the results of \cite[Chapter 6 and 7]{MR2560625}. In \cite{MR2560625}, the authors discuss stability and invariant measures for stochastic differential equations in the Hilbert valued setting, where the driving noise is a Hilbert valued $Q$-Wiener process. In our case, the noise is finite-dimensional, but solutions take values in certain real separable Hilbert spaces, the Hermite-Sobolev spaces (see subsection \ref{s1:topology} below). Throughout this article, solutions to the linear SPDE mentioned above, shall be considered in the notion of strong solutions. Another important distinction with the results of \cite{MR2560625} is the following. The solutions take values in a certain Hermite-Sobolev space and the equality in \eqref{SPDE} needs to be considered in a `lower regularity' Hermite-Sobolev space (see Theorem \ref{existence-uniqueness-solution} below).

\subsection{Topology on $\Sc$}\label{s1:topology}

It is known that $\Sc$ is a nuclear space and  that $\Sc = \cap_{p \in \R} \Sc_p$,  where the real separable Hilbert spaces $\Sc_p$, the Hermite-Sobolev spaces, are equipped with norms $\| \cdot \|_p$, defined by the inner products 
\[\inpr[p]{f}{g}  = \sum_{n = 0}^\infty \sum_{k : |k| = n} (2 |k| + d) ^ {2p} \inpr[]{f}{h _ k} \inpr{g}{h _ k}, \; f , g \in \Sc.\]
Here,  $ k = ( k _ 1 , k _ 2 , \cdots , k_d )$ denotes multi-indices with $ k _ j $'s being non-negative integers, $ |k| = k_1 + k_2 + \cdots + k_d $ and $ \inpr[]{\cdot}{\cdot}$ is the usual inner product in $L^2( \R^d , dx)$. 
In the above, $\{ h_k \}_{ k = 0 } ^ \infty$ is an orthonormal basis in $L^2(\R^d, dx)$ given by the Hermite functions (for $d=1, h_k(t) = (2^k k!\sqrt{\pi})^{-\frac{1}{2}}\exp{\{-\frac{t^2}{2}}\}H_k(t)$, with $H_k(t)$ being the Hermite polynomials (See  \cite{MR771478})).  We also  have $\Sc'=\cup_{p\in\R}\Sc_{p}$. Note that for all $p \in \R$, $\Sc_p$ is the completion of $\Sc$ in $\|\cdot\|_p$ and $\Sc_p'$ is isometrically isomorphic  with $\Sc_{-p}\; \forall p> 0$.

\subsection{Linear operators $L$ and $A$}\label{S1:operators}

 Let $\sigma=(\sigma_{ij})$ be  $\R^{d\times d}$ - valued  and $b=(b_1,b_2,\cdots,b_d)$ be $\R^d$ - valued  smooth functions with bounded derivatives defined on $\R^d$.

Let $L:\Sc' \to \Sc'$ and $ A_i:\Sc' \to \Sc',\; \forall i=1,2,\cdots,d $ be linear operators defined as
\begin{align} \label{Operators}
    A_i(\phi)=& -\sum_{j=1}^d\partial_j(\sigma_{ji}\phi) & L(\phi)=\frac{1}{2}\sum_{i,j=1}^d\partial^2_{ij}((\sigma \sigma^t)_{ij}\phi)+\sum_{i=1}^d\partial_i (b_i\phi) \;,\forall \phi\in\Sc.
\end{align}

 The distributional derivative operators $\partial_i: \Sc_p \to \Sc_{p}, i = 1, \cdots, d$ are densely defined closed linear unbounded operators. The following relation is well known \cite[Chapter 1]{MR1215939},
 \[ \partial_i h _ n = \sqrt{\frac{n_i}{2}} h_{n-e_i} - \sqrt{\frac{n_i+1}{2}}h _ {n+e_i}, \; \forall \; n=(n_1,n_2,\cdots,n_d), \; n_i \geq 0\] 
 where $\{e_1, \cdots,e_d\}$ is the standard basis in $\R^d$. Consequently, we have $\partial_i: \Sc_p \to \Sc_{p-\frac{1}{2}}$ are bounded linear operators.

 Let $\sigma_{ij}: \R \to \R, i,j = 1, \cdots, d$ be smooth functions with bounded derivatives. The following relation is well known \cite[Chapter 1]{MR1215939},
 \[ x_i h _ n = \sqrt{\frac{n_i}{2}} h_{n-e_i} + \sqrt{\frac{n_i+1}{2}}h _ {n+e_i}, \; \forall \; n=(n_1,n_2,\cdots,n_d), \; n_i \geq 0\] 
 By using above relation, it can be shown that the multiplication operators $\sigma_{ij}: \Sc_p \to \Sc_{p-1}, i,j = 1, \cdots, d$  are bounded linear operators \cite[Proposition 3.2]{MR2373102}.
 
 If $ \phi\in \Sc_p$, then $A_i(\phi), L(\phi)\in \Sc_{p-2} , \forall i$. In fact, $\left\|A(\phi)\right\|_{HS(p-2)}^2 := \sum_{i=0}^d\|A_i\phi\|_{p-2}^2 \leq C_1 \left\|\phi\right\|_p^2$ and $\left\|L(\phi)\right\|_{p-2} \leq C_2 \left\|\phi\right\|_{p}, \phi\in\Sc$  for some positive constants $C_1$ and $C_2$ (see  \cite{MR2373102}).
 
Now we recall the notion of Monotonicity inequality.

\begin{definition}
    We say that the  pair of linear operators $(L, A)$ satisfies the Monotonicity inequality in $\|\cdot\|_p$, if we have
\begin{equation}\label{Monotoniticity-inequality}
2\inpr[p]{\phi}{L\phi} + \|A\phi\|_{HS(p)}^2\leq C\|\phi\|^2_p,\quad \forall \phi\in \Sc
\end{equation}
where $C=C(\sigma,b,p,d) > 0$ is a positive constant.
\end{definition}

The Monotonicity inequality for $(L, A)$ has been proved explicitly in the following cases.

    \begin{theorem}
    For every $p\in \R$, the monotonicity inequality \eqref{Monotoniticity-inequality} for the pair $(L,A)$ holds for the following cases
    \begin{enumerate}
        \item \cite[Theorem 2.1]{MR2590157}. $\sigma \in \R^{ d \times d}$ and $b \in \R^d$
        \item\cite[Theorem 4.6]{MR3331916}. $\sigma\in \R^{d\times d}$ and $b(x):=b_0+Mx$ on $\R^d$ for fixed $b_0\in \R^d, M\in\R^{d\times d}$.
    \end{enumerate}
\end{theorem}

  In the rest of the article, we assume that $p \in \R$ is fixed and
  \begin{equation}\tag{Hyp}\label{Hyp}
\text{the pair of linear operators } (L, A) \text{ satisfies the Monotonicity inequality in } \|\cdot\|_p \text{ and } \|\cdot\|_{p-2}.
  \end{equation}
We shall denote the relevant constant appearing in the Monotonicity inequality in $\|\cdot\|_{p-2}$ by $C_0$. By \eqref{Hyp}, the pair $(L - \alpha, A)$ also satisfies the Monotonicity inequality in $\|\cdot\|_{p-2}$. As a consequence of \cite[Theorem 1]{MR2479730}, the next result follows.

  \begin{theorem}[{\cite[Theorem 1]{MR2479730}}]\label{existence-uniqueness-solution}
      Under \eqref{Hyp}, the SPDE \eqref{SPDE} has an unique $ \Sc_p$-valued strong solution $\{X_t^x, t\geq 0\}$ for every deterministic initial condition $X_0=x \in \Sc_p$, where the equality holds in $\Sc_{p-2}$.
  \end{theorem}
  
By standard arguments (see, for example, \cite[Theorem 4.8]{MR2560625}), it follows that the solution $\{X_t^x, t\geq0\}$ is a Markov process, and the corresponding semigroup $\{P_t, t\geq 0\} $ has Feller Property.

\subsection{Main Results: Exponential Stability and Invariant Measures}\label{S1:results}

In this subsection, we introduce the notions of exponential stability and invariant measures used in this article and then state the main results. As pointed out in the introduction, these results are analogous to the results of \cite[Chapter 6 and 7]{MR2560625}. The next definition is similar to \cite[Definition 6.1]{MR2560625}.

\begin{definition}
    Fix $q\leq p$. We say that the $\Sc_p$- valued  solution $\{X_t^x,t \geq 0\}$ of \eqref{SPDE} is exponentially stable in $\Sc_q$  in mean square sense if there exist positive constants $c, \beta$ such that 
    \[\Exp\|X^x_t\|^2_q\leq c\|x\|^2_qe^{-\beta t},\; \forall x\in \Sc_p\]
\end{definition}

Under the above consideration, exponential stability has been shown in the following proposition.
\begin{proposition}\label{exp stability}
     Under \eqref{Hyp}, the $\Sc_p$- valued solution of  SPDE \eqref{SPDE} is exponentially stable in 
    $\Sc_{p-2}$,  if $2\alpha > C_0$.
\end{proposition}
    
Now we introduce the concept of  Invariant measure  for the SPDE \eqref{SPDE}. This is analogous to \cite[Definition 7.4]{MR2560625}.

\begin{definition}
Fix $q \leq p$. We say that a probability measure $\mu$ on $\Sc_q$ is invariant for the SPDE \eqref{SPDE} with the related Feller semigroup  $\{P_t, t\geq 0\} $, if for all Borel sets  $A$ in $\Sc_q$,
    \[\mu(A)=\int_{\Sc_q}P(t,y,A)\mu(dy)\]
    or equivalently, since $\Sc_q$ is a Polish space, then for  all bounded continuous real valued function $f$ on $\Sc_q$,
    \[\int_{ \Sc_q }( P_t f ) d \mu = \int_{ \Sc_q } f  d \mu . \] 
\end{definition}

 We now state the main results in next two theorems.
 
\begin{theorem}\label{unique-in-q-2}
    Fix $q < p$. Suppose \eqref{Hyp} holds and $2\alpha > C_0$. Then, there exists a unique invariant measure $ \mu $ on $\Sc_{ q - 2 }$ for the SPDE \eqref{SPDE}, provided the pair $(L - \alpha, A)$ satisfies the Monotonicity inequality in $\|\cdot\|_{q - 2}$.
\end{theorem}

The proofs of the results have been discussed in Section 2.

\section{Proofs of Main results}\label{S2:proofs}
In this section, we discuss the proofs.

\begin{proof}[Proof of Proposition \ref{exp stability}]
    Let $\{X_t^x,t \geq 0\}$ be the $\Sc_p$- valued  solution of SPDE \eqref{SPDE}, we have
    \[X_t^x = x+ \int_0^t\left ( L-\alpha\right)(X^x_s)ds +\int^t_0 A(X^x_s) \, dB_s\] 
     and the equality holds in $\Sc_{p-2}$. By applying It\^o rule, Monotonicity inequality and  taking expectation,
   \begin{align*}
      \Exp \left \| X_t^x \right \|_{p-2}^2 = &  \left\| x \right\|_{p-2}^2 + \Exp\int_0^t\left ( 2\inpr[p-2]{X^x_s}{(L-\alpha)( X^x_s)}+ \left\| A  X^x_s \right\|_{HS(p-2)}^2 \right) ds \\
      =&  \left\| x \right\|_{p-2}^2 + \int_0^t \Exp \left ( 2\inpr[p-2]{X^x_s}{(L-\alpha)( X^x_s)}+ \left\|A X^x_s\right\|_{HS(p-2)}^2\right)ds\\
      \leq&  \left\| x \right\|_{p-2}^2 +\left( C_0-2 \alpha \right)  \int_0^t  \Exp  \left\|  X^x_s \right\|_{p-2}^2 ds
    \end{align*}
    By Gronwall's inequality, we have
    \begin{equation}\label{L2-norm-bound}
       \Exp \left\|  X_t^x \right\|_{ p - 2 }^2 \leq \left\| x \right\|_{ p-2 }^2 e^{( C_0 - 2 \alpha ) t}. 
    \end{equation}
    Hence $\{X_t^x, t \geq 0\}$ is exponentially stable in $\Sc_{ p - 2 }$.
\end{proof}

We now discuss a few technical lemmas required in the proofs of main results.

\begin{lemma}\label{compact-emmbadding}
    The inclusion $i:\Sc_p \to \Sc_q$ is compact if $q<p$.
\end{lemma}
\begin{proof}
 Any $x \in \Sc_p$ has an expansion $\sum_{ j = 0 }^\infty \sum_{k : |k| = j} \inpr{x}{ h_k } h_k $
   . Then we consider the projection map $T_n:\Sc_p \to \Sc_q$ on the span of $\{h_k : 0 \leq |k| \leq n\}$   as follows,
    \[T_n(x)=\sum_{j = 0}^n \sum_{k : |k| = j}\inpr{x}{h_k}h_k,\; \forall x\in \Sc_p\]
    Since $T_n$'s are finite rank operators, then compact. 
    Now
    \begin{align*}
        \|T_n(x) - i(x) \|_q^2=&\sum_{j = n+1}^\infty \sum_{k : |k| = j}\inpr[]{x}{h_k}^2\|h_k\|_q^2\\
        =&\sum_{j = n+1}^\infty \sum_{k : |k| = j}\inpr[]{x}{h_k}^2(2|k|+d)^{2q}\\
         =&(2n+d)^{-2(p-q)}\sum_{j = n+1}^\infty \sum_{k : |k| = j}\inpr[]{x}{h_k}^2(2|k|+d)^{2q}(2n+d)^{2(p-q)}\\
          \leq &(2|n|+d)^{-2(p-q)}\sum_{j = n+1}^\infty \sum_{k : |k| = j}\inpr[]{x}{h_k}^2(2|k|+d)^{2q}(2|k|+d)^{2(p-q)}\\
            = &(2n+d)^{-2(p-q)}\sum_{j = n+1}^\infty \sum_{k : |k| = j}\inpr[]{x}{h_k}^2(2|k|+d)^{2p}\\
             \leq &(2n+d)^{-2(p-q)}\|x\|_p^2.
    \end{align*}
    Therefore, $\frac{\|T_n(x)- i(x) \|_q}{\|x\|_p} \leq \frac{1}{(2n+d)^{p-q}},\; \forall (0\neq) x\in \Sc_p$ and hence, we have $T_n\to i$ as $n\to \infty$. Therefore, using \cite[Corollary 6.2]{MR2759829}, the embedding is compact.
\end{proof}

\begin{remark}
    The Identity operator $Id: \Sc_p \to \Sc_p$ can not be compact \cite[Theorem 6.5 (Riesz)]{MR2759829}
\end{remark}

\begin{proof}[Proof of Theorem \ref{unique-in-q-2}]
    For any $R > 0$, we have
    \[ \pp(\|X_t^x\|_{p-2}\geq R)\leq \frac{1}{R^2} \Exp\|X_t^x\|^2_{p-2},\]
and hence by \eqref{L2-norm-bound} and using $2\alpha > C_0$, we have
\[  \frac{1}{T}\int^T_0\pp(\|X_t^x\|_{p-2}\geq R)dt \leq \frac{1}{TR^2}\int^T_0 \Exp\|X_t^x\|^2_{p-2}dt \leq \frac{\left\| x \right\|_{p-2}^2}{R^2} \frac{1}{T}\int^T_0  e^{(C_0-2\alpha)t}dt \leq \frac{\left\| x \right\|_{p-2}^2}{R^2}.\]
Hence,
\[    \lim_{R\to \infty}\sup_T \frac{1}{T}  \int^T_0\pp(\|X_t^x\|_{p-2}\geq R)dt=0\]
 Hence, given any $\epsilon >0$, there exists an $R_\epsilon\geq 0$ such that 
 \[\sup_T \frac{1}{T}  \int^T_0 \pp(\|X_t^x\|_{p-2}\geq R_\epsilon)dt<\epsilon\]
 By Lemma \ref{compact-emmbadding}, the embedding $\Sc_{p-2} \hookrightarrow \Sc_{q-2}$ is compact for $q<p$, and hence the set $\{x\in \Sc_{p-2}:\|x\|_{p-2}\leq R_\epsilon\}$ is compact in $\Sc_{q-2}$. There exists an invariant measure $\mu $  on $\Sc_{q-2}$ for SPDE \eqref{SPDE}   by Prokhorov's theorem \cite[Corollary 7.4]{MR2560625}  and $\mu$ is unique using arguments similar to \cite[Theorem 7.13]{MR2560625}. Existence of a strong solution $\{X_t, t \geq 0\}$ of \eqref{SPDE} with $X_0 \sim \mu$ follows from the Monotonicity inequality for the pair $(L - \alpha, A)$ in $\|\cdot\|_{q - 2}$ (see \cite[Theorem 1]{MR2479730}).
\end{proof}

\subsection*{Acknowledgement:}  Arvind Kumar Nath would like to acknowledge the fact that he was supported
 by the University Grants Commission (Government of India) Ph.D research Fellowship.

    \bibliographystyle{plain}
\bibliography{ref}

\begin{thebibliography}{10}

\bibitem{MR3331916}
Suprio Bhar and B.~Rajeev.
\newblock Differential operators on {H}ermite {S}obolev spaces.
\newblock {\em Proc. Indian Acad. Sci. Math. Sci.}, 125(1):113--125, 2015.

\bibitem{MR2759829}
Haim Brezis.
\newblock {\em Functional analysis, {S}obolev spaces and partial differential equations}.
\newblock Universitext. Springer, New York, 2011.

\bibitem{MR3236753}
Giuseppe Da~Prato and Jerzy Zabczyk.
\newblock {\em Stochastic equations in infinite dimensions}, volume 152 of {\em Encyclopedia of Mathematics and its Applications}.
\newblock Cambridge University Press, Cambridge, second edition, 2014.

\bibitem{MR1681462}
Gerald~B. Folland.
\newblock {\em Real analysis}.
\newblock Pure and Applied Mathematics (New York). John Wiley \& Sons, Inc., New York, second edition, 1999.
\newblock Modern techniques and their applications, A Wiley-Interscience Publication.

\bibitem{MR2479730}
L.~Gawarecki, V.~Mandrekar, and B.~Rajeev.
\newblock Linear stochastic differential equations in the dual of a multi-{H}ilbertian space.
\newblock {\em Theory Stoch. Process.}, 14(2):28--34, 2008.

\bibitem{MR2590157}
L.~Gawarecki, V.~Mandrekar, and B.~Rajeev.
\newblock The monotonicity inequality for linear stochastic partial differential equations.
\newblock {\em Infin. Dimens. Anal. Quantum Probab. Relat. Top.}, 12(4):575--591, 2009.

\bibitem{MR2560625}
Leszek Gawarecki and Vidyadhar Mandrekar.
\newblock {\em Stochastic differential equations in infinite dimensions with applications to stochastic partial differential equations}.
\newblock Probability and its Applications (New York). Springer, Heidelberg, 2011.

\bibitem{MR0916950}
Xuan~Da Hu.
\newblock Boundedness and invariant measures of semilinear stochastic evolution equations.
\newblock {\em Nanjing Daxue Xuebao Shuxue Bannian Kan}, 4(1):1--14, 1987.

\bibitem{MR0738933}
Akira Ichikawa.
\newblock Semilinear stochastic evolution equations: boundedness, stability and invariant measures.
\newblock {\em Stochastics}, 12(1):1--39, 1984.

\bibitem{MR771478}
Kiyosi It\^{o}.
\newblock {\em Foundations of stochastic differential equations in infinite-dimensional spaces}, volume~47 of {\em CBMS-NSF Regional Conference Series in Applied Mathematics}.
\newblock Society for Industrial and Applied Mathematics (SIAM), Philadelphia, PA, 1984.

\bibitem{MR1121940}
Ioannis Karatzas and Steven~E. Shreve.
\newblock {\em Brownian motion and stochastic calculus}, volume 113 of {\em Graduate Texts in Mathematics}.
\newblock Springer-Verlag, New York, second edition, 1991.

\bibitem{MR1396756}
Ruifeng Liu and V.~Mandrekar.
\newblock Ultimate boundedness and invariant measures of stochastic evolution equations.
\newblock {\em Stochastics Stochastics Rep.}, 56(1-2):75--101, 1996.

\bibitem{MR2020294}
Philip~E. Protter.
\newblock {\em Stochastic integration and differential equations}, volume~21 of {\em Applications of Mathematics (New York)}.
\newblock Springer-Verlag, Berlin, second edition, 2004.
\newblock Stochastic Modelling and Applied Probability.

\bibitem{MR2373102}
B.~Rajeev and S.~Thangavelu.
\newblock Probabilistic representations of solutions of the forward equations.
\newblock {\em Potential Anal.}, 28(2):139--162, 2008.

\bibitem{MR1157815}
Walter Rudin.
\newblock {\em Functional analysis}.
\newblock International Series in Pure and Applied Mathematics. McGraw-Hill, Inc., New York, second edition, 1991.

\bibitem{MR1215939}
Sundaram Thangavelu.
\newblock {\em Lectures on {H}ermite and {L}aguerre expansions}, volume~42 of {\em Mathematical Notes}.
\newblock Princeton University Press, Princeton, NJ, 1993.
\newblock With a preface by Robert S. Strichartz.

\bibitem{MR2296978}
Fran\c{c}ois Tr\`eves.
\newblock {\em Topological vector spaces, distributions and kernels}.
\newblock Dover Publications, Inc., Mineola, NY, 2006.
\newblock Unabridged republication of the 1967 original.

\end{thebibliography}
\end{document}